\documentclass[11pt,reqno]{amsart}
\usepackage{graphicx}
\usepackage{color}
\usepackage{amsmath,amssymb,amsthm,amsfonts}
\usepackage{graphicx}
\usepackage{color}
\usepackage{multicol}
\def\R{\mathbb{R}}
\makeatletter

\newcommand{\Rmnum}[1]{\expandafter\@slowromancap\romannumeral #1@}
\makeatother

\newtheorem{thm}{Theorem}[section]

\newtheorem{lemma}[thm]{Lemma}

\newtheorem{theorem}[thm]{Theorem}

\usepackage[numbers,sort&compress]{natbib}
\begin{document}
\author{Hai-Yang Jin}
\address{Department of Mathematics, South China University of Technology, Guangzhou 510640, China}
\email{mahyjin@scut.edu.cn}

\author{Zhi-An Wang}
\address{Department of Applied Mathematics, Hong Kong Polytechnic University, Hung Hom,
 Hong Kong}
\email{mawza@polyu.edu.hk}

\title[ The Keller-Segel system with logistic growth and signal dependent diffusion ]{ The Keller-Segel system with logistic growth and signal-dependent motility}

\begin{abstract}
The paper is concerned with the following chemotaxis system with nonlinear motility functions
\begin{equation}\label{0-1}\tag{$\ast$}
\begin{cases}
u_t=\nabla \cdot (\gamma(v)\nabla u- u\chi(v)\nabla v)+\mu u(1-u), &x\in \Omega, ~~t>0,\\
 0=\Delta v+ u-v,& x\in \Omega, ~~t>0,\\
u(x,0)=u_0(x), & x\in \Omega,
\end{cases}
\end{equation}
with homogeneous Neumann boundary conditions in a bounded domain $\Omega\subset \R^2$ with smooth boundary, where the motility functions $\gamma(v)$ and $\chi(v)$ satisfy the following conditions
\begin{itemize}
\item {\color{black}$(\gamma,\chi)\in [C^2[0,\infty)]^2$} with $\gamma(v)>0$ and {\color{black} $\frac{|\chi(v)|^2}{\gamma(v)}$ is bounded for all $v\geq 0$.}
%for all $v\geq 0$ and $\lim\limits_{v\to\infty}\frac{|\chi(v)|^2}{\gamma(v)}$ exists.
\end{itemize}
By employing the method of energy estimates , we  establish the  existence of globally bounded solutions of \eqref{0-1} with $\mu>0$ for any $u_0 \in W^{1, \infty}(\Omega)$. Then based on a Lyapunov function, we show that all solutions $(u,v)$ of \eqref{0-1} will exponentially converge to the unique constant steady state $(1,1)$ provided $\mu>\frac{K_0}{16}$ with $K_0=\max\limits_{0\leq v \leq \infty}\frac{|\chi(v)|^2}{\gamma(v)}$.
\end{abstract}



   % acknowledge support, etc

\maketitle

\numberwithin{equation}{section}

\maketitle

\numberwithin{equation}{section}
\section{Introduction and main results}
In this paper, we consider the following chemotaxis model with density-dependent motilities
\begin{equation}\label{OR}
\begin{cases}
u_t=\nabla \cdot (\gamma(v)\nabla u- u\chi(v)\nabla v)+\mu u(1-u),&x\in\Omega,\ \ t>0, \\
\tau v_t=\Delta v+  u-v,&x\in\Omega,\ \ t>0,\\
\frac{\partial u}{\partial \nu}=\frac{\partial v}{\partial \nu}=0,&x\in \partial \Omega,t>0,\\
u(x,0)=u_0(x),v(x,0)= v_0(x) & x\in \Omega,
\end{cases}
\end{equation}
where $\Omega\subset\R^n (n\geq 2)$ is a bounded domain with smooth boundary, $u(x,t)$ denotes the cell density and $v(x,t)$ is the chemical concentration,  $\mu\geq 0$ and $\tau=\{0,1\}$. The prominent feature of \eqref{OR} compared to the classical chemotaxis model is that both the undirected motility (diffusion) and directed motility (chemotaxis) of cells depend on the chemical concentration. The system \eqref{OR} has several important applications. When $\mu=0$, the system \eqref{OR} has been firstly derived by Keller and Segel in \cite{KS1} to describe the aggregation phase of amoeba cells in response to the chemical signal cAMP emitted by themselves, where the motility functions $\gamma(v)>0$ and $\chi(v)$ are correlated by the following proportionality relation
\begin{equation}\label{r1}
\chi(v)=(\alpha-1)\gamma'(v),
\end{equation}
with $\alpha$ denoting the ratio of effective body length to step size, and $\gamma'(v)<0 \ (\text{resp.} >0)$ if the diffusive motility decreases (\text{resp.} increases) with respect to the chemical concentration. As mentioned in \cite{KS1}, although the motility coefficient $\gamma(v)$ is positive, the chemotactic  motility coefficient  $\chi(v)$ may be positive or negative depending on the signs of $(\alpha-1)$ and $\gamma'(v)$.

When both $\gamma(v)$ and $\chi(v)$ are constant, \eqref{OR} is called the minimal chemotaxis system which has been extensively studied  in the literature from various aspects including boundedness, blow-up, large-time behavior and pattern formation of solutions  (cf. \cite{X-JDE-2015, OTYM,W-CPDE-2010, TeW-CPDE-2007, LM-DCDSA-2016, Winkler-JDE-2014,BW-2016, PH-PD-2011,MOW, TeW-CPDE-2007,KOS,Lan-DCDSB-2015} and reference therein). When $\gamma(v)$ is constant and $\chi(v)=1/v$, the system \eqref{OR} with $\mu=0$ has been studied recently in a number of interesting works (see \cite{FS-Nonlinearity-2016, W-M2AS-2011} and references therein).  However, if $\gamma(v)$ is non-constant, the results of \eqref{OR} are very limited. The few existing results are mainly focused on  the special case $\chi(v)=-\gamma'(v)$, (i.e. $\alpha=0$), which reduces the system \eqref{OR} to
 \begin{equation}\label{OR-1}
\begin{cases}
u_t=\Delta (\gamma(v) u)+\mu u(1-u),&x\in\Omega,\ t>0, \\
\tau v_t=\Delta v+  u- v,&x\in\Omega,\ t>0,\\
\frac{\partial u}{\partial \nu}=\frac{\partial v}{\partial \nu}=0,&x\in \partial \Omega,t>0,\\
u(x,0)=u_0(x),v(x,0)= v_0(x) & x\in \Omega.
\end{cases}
\end{equation}
Essentially \eqref{OR-1} with $\mu>0$ has been used in \cite{Fu} to justify that the bacterial with density-suppressed motility (i.e., $\gamma'(v)<0$) can  produce the stripe pattern formation observed in the experiment of \cite{Liu}. Several results on the deduced system \eqref{OR-1} are then available as will be recalled below.

When $\mu=0$ (no cell growth), it was proved in \cite{YK} that the system \eqref{OR-1} with $\tau=1$ and $\gamma(v)=c_0/v^{k}(k>0)$ admits global classical solutions in any dimensions for small constant $c_0>0$. Recently, the smallness assumptions of $c_0$ was removed in \cite{AY-Nonlinearity-2019} for the parabolic-elliptic case of \eqref{OR-1} (i.e., $\tau=0$) for any $0<k<\frac{2}{n-2}$. {\color{black}Moreover, based on the phase plane analysis and bifurcation analysis, the existence and analytical approximation of non-constant stationary were established in one dimension \cite{Xia-Ma}.} By assuming that $\gamma(v)$ has a positive lower and upper bound {\color{black}(i.e. $\delta_1\leq \gamma(v)\leq \delta_2$ for some positive constants $\delta_1, \delta_2$)}, the global classical solution in two dimensions and global weak solution in three dimensions of \eqref{OR-1} with $\mu=0$ are obtained in \cite{TW-M3AS-2017}. Recently,  it is proved in {\color{black}\cite{JW-PAMS-2019,Fujie-Jiang}} that if $\gamma(v)=e^{-\chi v}$ there exists a critical mass $m_*=\frac{4\pi}{\chi}$ such that the solution of \eqref{OR-1} with $\mu=0$ exists globally with uniform-in-time bound if $\int_\Omega u_0dx<m_*$ and blows up if $\int_\Omega u_0dx>m_*$. Turing to the case $\mu>0$, there are several results below. When $\gamma(v)$ is a decreasing step-wise constant function, the dynamics of discontinuity interface of solutions is studied in \cite{Smith} in one dimension. In two dimensional spaces, the global boundedness of solutions of  \eqref{OR-1} with $\tau=1$ was established in \cite{JKW-SIAP-2018} under the following hypotheses on the motility function $\gamma(v)$:
\begin{itemize}
\item[(H0)] $\gamma(v)\in C^3([0,\infty)), \gamma(v)>0\ \  \mathrm{and}~\gamma'(v)<0 \ \ \mathrm{on} ~[0,\infty)$, $\lim\limits_{v \to \infty}\gamma(v)=0$,\ \ $\lim\limits_{v \to \infty}\frac{\gamma'(v)}{\gamma(v)}$ exists.
\end{itemize}
It was further shown in \cite{JKW-SIAP-2018} that the constant steady state $(1,1)$ is globally asymptotically stable provided
$
\mu>\frac{K_0}{16}$ with $K_0=\max\limits_{0\leq v \leq \infty}\frac{|\gamma'(v)|^2}{\gamma(v)}.
$
{\color{black}Similar }results have been extended to higher dimensions ($n\geq 3$) in \cite{WW-JMP-2019} for large $\mu>0$. The existence/{\color{black}nonexistence} of nonconstant steady states  of \eqref{OR-1} was recently studied in \cite{MPW-PD-2019}. Moveover, the global existence of solutions of \eqref{OR-1} with $\tau=0$ was obtained in \cite{Fujie-Jiang1} without  the condition ``$\lim\limits_{v \to \infty}\frac{\gamma'(v)}{\gamma(v)}$ exists'' in (H0).

In summary, for the chemotaxis system \eqref{OR}-\eqref{r1} with density-dependent motility, the results are available only for the special case $\alpha=0$ with various hypotheses on the motility function $\gamma(v)$ as recalled above for \eqref{OR-1}. Therefore there are various interesting questions remaining open. The following questions comprise the motivation of this paper.
\begin{itemize}
\item[(Q1)] So far no results of \eqref{OR}-\eqref{r1} are available for $\alpha\ne 0$ in the prescribed proportionality relation \eqref{r1}. Furthermore as remarked in \cite{KS1}, the prescribed proportionality \eqref{r1} between the motility functions $\gamma(v)$ and $\chi(v)$ is derived based on assumption that the cell step size is constant and the total step frequency is solely determined by the mean concentration of the chemical. However, $\chi(v)$ would  no longer be simple proportional to $\gamma'(v)$  if both step size and total step frequency were permitted to vary with the chemical concentration. Hence, it would be meaningful and interesting to study the system \eqref{OR} with more general  $\gamma(v)$ and $\chi(v)$ beyond the proportionality \eqref{r1}.
\item[(Q2)] The previous results as recalled above are mostly restricted to the case $\gamma'(v)<0$ or some special form of $\chi(v)$ (cf. \cite{AY-Nonlinearity-2019,JKW-SIAP-2018,WW-JMP-2019,JW-PAMS-2019}). However, as discussed in  \cite[Section 3]{KS1},  the cell motion may be more vigorous  at high concentrations than at low concentrations, which motives us to study the case $\gamma'(v)>0$ or even non-monotone $\gamma(v)$ so that the analytical results can cover more possible applications.
%\item[(Q3)]In the special case $\alpha=0$ of \eqref{r1} (namely $\chi(v)=-\gamma'(v)$) which reduces the system %\eqref{OR} to \eqref{OR-1}, the hypotheses (H0) requires the condition %$\lim\limits_{v\to\infty}\frac{\gamma'(v)}{\gamma(v)}$, which may cover a large class of functions like algebraically %or exponentially decay functions (see \cite{JKW-SIAP-2018}), but excludes some decay function like  %$\gamma(v)=e^{-e^v}$ which leads to $\lim\limits_{v\to\infty}\frac{\gamma'(v)}{\gamma(v)}=-\infty$. Hence the %question whether the condition $\lim\limits_{v\to\infty}\frac{\gamma'(v)}{\gamma(v)}$ exists in (H0) can be relaxed %or remove. %$\lim\limits_{v\to\infty}\frac{|\chi(v)|^2}{\gamma(v)}=\lim\limits_{v\to\infty}\frac{|\gamma'(v)|^2}{\gamma(v)}=0$. %In other words, the hypotheses (H1) covers a wider class of $\gamma(v)$ in the case of $\chi(v)=-\gamma'(v)$ and %$\gamma'(v)<0$.
\end{itemize}

Inspired by the above mentioned questions, in this paper we shall develop some first-hand results on the global boundedness and large time behavior of solutions to the system \eqref{OR} with general motility functions $\gamma(v)$ and $\chi(v)$.  Specifically we consider \eqref{OR} with $\tau=0$
\begin{equation}\label{1-1*}
\begin{cases}
u_t=\nabla \cdot (\gamma(v)\nabla u- u\chi(v)\nabla v)+\mu u(1-u),&x\in\Omega,\ \ t>0, \\
0=\Delta v+  u-v,&x\in\Omega,\ \ t>0,\\
\frac{\partial u}{\partial \nu}=\frac{\partial v}{\partial \nu}=0,&x\in \partial \Omega,t>0,\\
u(x,0)=u_0(x),v(x,0)= v_0(x) & x\in \Omega,
\end{cases}
\end{equation}
under the following assumptions on $\gamma(v)$ and $\chi(v)$:
\begin{itemize}
\item[(H1)]  ${\color{black}(\gamma,\chi)}\in [C^2[0,\infty)]^2$ with $\gamma(v)>0$ and {\color{black} $\frac{|\chi(v)|^2}{\gamma(v)}$ is bounded for all $v\geq 0$.}
 %for all $v\geq 0$ and $\lim\limits_{v\to\infty}\frac{|\chi(v)|^2}{\gamma(v)}$ exists.
\end{itemize}
The main results of this paper are the following.
\begin{theorem}\label{GB}
Let $\Omega$ be a bounded domain  in $\R^2$ with smooth boundary and the hypotheses (H1) hold. Suppose that $u_0\in W^{1,\infty}(\Omega)$ with $u_0\geq 0(\not \equiv 0)$. Then
the problem (\ref{1-1*}) has a unique global classical solution $(u,v)\in [C([0,\infty)\times \bar{\Omega})\cap C^{2,1}((0,\infty)\times\bar{\Omega})]\times C^{2,1}((0,\infty)\times\bar{\Omega})$ satisfying $u,v>0$ for all $t>0$ and
\begin{equation*}\label{1-1}
\|u(\cdot,t)\|_{L^\infty(\Omega)}+\|v(\cdot,t)\|_{W^{1,\infty}(\Omega)}\leq C_1 \ \ \mathrm{for \ all}\  t>0,
\end{equation*}
where $C_1>0$ is a constant independent of $t$. Furthermore,  if $\mu>\frac{K_0}{16}$ with $K_0=\max\limits_{0\leq v \leq \infty}\frac{|\chi(v)|^2}{\gamma(v)}$, then  there exist two positive constants $C_2$ and $\delta$ such that
\begin{equation*}
\|u(\cdot,t)-1\|_{L^\infty(\Omega)}+\|v(\cdot,t)-1\|_{L^\infty(\Omega)}\leq  C_2e^{-\delta t}.
\end{equation*}
\end{theorem}
The results in Theorem \ref{GB} not only address the questions raised in (Q1) and (Q2), but also {\color{black}improve} the existing results on the specialized system \eqref{OR-1} where $\chi(v)=-\gamma'(v)$. Indeed with $\alpha=0$ in \eqref{r1} with $\gamma'(v)<0$, {\color{black} one can check that  ``$\lim\limits_{v\to\infty}\frac{\gamma'(v)}{\gamma(v)}$ exists''  in (H0) is a stronger condition than ``$\frac{|\chi(v)|^2}{\gamma(v)}$ is bounded for all $v\geq 0$" %``$\lim\limits_{v\to\infty}\frac{|\chi(v)|^2}{\gamma(v)}=\lim\limits_{v\to\infty}\frac{|\gamma'(v)|^2}{\gamma(v)}$ exists''
in (H1)}. For example, if $\gamma(v)=e^{-e^v}$, then $\lim\limits_{v\to\infty}\frac{\gamma'(v)}{\gamma(v)}=-\infty$ but {\color{black} $\frac{|\chi(v)|^2}{\gamma(v)}=\frac{|\gamma'(v)|^2}{\gamma(v)}=e^{(2v-e^v)}\leq e^{2(\ln 2-1)}$ for any $v\geq 0$}.
%$\lim\limits_{v\to\infty}\frac{|\chi(v)|^2}{\gamma(v)}=\lim\limits_{v\to\infty}\frac{|\gamma'(v)|^2}{\gamma(v)}=0$.
We remark the same results of \eqref{OR-1} with $\tau=0$ as in \cite{JKW-SIAP-2018} for $\tau=1$ are obtained in \cite{Fujie-Jiang1} without  the condition ``$\lim\limits_{v\to\infty}\frac{\gamma'(v)}{\gamma(v)}$ exists'' in (H0), where the methods developed therein essentially rely on the monotonicity of $\gamma(v)$ and the proportionality relation $\chi(v)=-\gamma'(v)$ and hence are inapplicable to our present problem where we consider more general $\gamma(v)$ and $\chi(v)$ without such restrictions.

\section{Local existence and Preliminaries}
In what follows, without confusion,  we shall abbreviate $\int_\Omega fdx$ as $\int_\Omega f$ and $\|f\|_{L^2(\Omega)}$ as $\|f\|_{L^2}$ for simplicity. Moreover, we shall use $c_i (i=1,2,3, \cdots$) to denote a generic constant which may vary in the context.
 The existence of local solutions of \eqref{1-1*} can be proved  by Schauder fixed point theorem as illustrated in \cite[Lemma 2.1]{JKW-SIAP-2018} for the system \eqref{OR-1} with $\tau=1$, we omit the details for brevity.
\begin{lemma}[Local existence]\label{LS}
Let $\Omega $ be a bounded domain in $\R^2$ with smooth boundary and the hypothesis (H) hold. Assume $ u_0\in W^{1,\infty}(\Omega)$ with $u_0\geq 0(\not \equiv 0)$.  Then there exists $T_{max}\in(0,\infty]$ such that the  problem \eqref{1-1*} has a unique classical solution  $(u,v) \in [C([0,\infty)\times \bar{\Omega})\cap C^{2,1}((0,\infty)\times\bar{\Omega})]\times C^{2,1}((0,\infty)\times\bar{\Omega})$ satisfying $u, v>0$ for all $t>0$. Moreover, we have
\begin{equation*}\label{bc}
Either \  \ T_{max}=\infty, \ or\ \  \limsup\limits_{t\nearrow T_{max}}(\|u(\cdot,t)\|_{L^\infty}+\|v(\cdot,t)\|_{W^{1,\infty}})=\infty.
\end{equation*}
\end{lemma}
\begin{lemma}\label{L1u} Let $(u,v)$ be the solution of system \eqref{1-1*}. Then it holds that
\begin{equation}\label{L1u-1}
\int_\Omega u\leq m_*:=\max\{\|u_0\|_{L^1},|\Omega|\},\ \ \mathrm{for\ all }\ \ t\in(0,T_{max}).
\end{equation}
\end{lemma}
\begin{proof}
We integrate the first equation of  (\ref{1-1*}) over $\Omega$ to have
\begin{equation*}\label{2L1-4}
\frac{d}{dt}\int_\Omega u+\mu \int_\Omega u^2=\mu\int_\Omega u,\ \ \mathrm{for\  all }\ \ t\in(0,T_{max}),
\end{equation*}
which, together with $\int_\Omega u^2\geq \frac{1}{|\Omega|}\left(\int_\Omega u\right)^2$, gives
 \begin{equation*}\label{2L1*}
 \frac{d}{dt}\int_\Omega u\leq \mu \int_\Omega u-\frac{\mu}{|\Omega |}\left(\int_\Omega u\right)^2,\ \ \mathrm{for\ all }\ \ t\in(0,T_{max})
 \end{equation*}
 and hence \eqref{L1u-1} follows.
 \end{proof}

\section{Proof of Theorem \ref{GB}}
In this section, we shall prove Theorem \ref{GB}. First, we show the global existence of uniformly-in-time bounded solutions.
\subsection{Boundedness of solutions}
\begin{lemma} \label{L2}
Suppose the conditions in Theorem \ref{GB} hold. Then there exists a constant $C>0$ independent of $t$ such that
\begin{equation}\label{L2-1}
\|u\ln u\|_{L^1}\leq C  \ \mathrm{for\  all }\ \ t\in(0,T_{max})
\end{equation}
and
\begin{equation}\label{L2-1*}
\|\nabla v\|_{L^2}\leq C \  \ \mathrm{for\  all }\ \ t\in(0,T_{max}).
\end{equation}
\end{lemma}
\begin{proof}
Multiplying the first equation of  \eqref{1-1*} by $\ln u$, and integrating the result  by part, one has
\begin{equation}\label{L2-2}
\begin{split}
\frac{d}{dt}\left(\int_\Omega u\ln u-\int_\Omega  u\right)+\int_\Omega \gamma (v)\frac{|\nabla u|^2}{u}=\int_\Omega \chi(v)\nabla v\cdot \nabla u+\mu\int_\Omega u\ln u-\mu \int_\Omega u^2 \ln u.
\end{split}
\end{equation}
From the assumptions in (H1), we can find a constant  $K>0$ such that
\begin{equation}\label{rvK}
\frac{|\chi(v)|^2}{\gamma(v)}\leq K \ \ \mathrm{for \  all \ \ }  v\geq 0.
\end{equation}
Using the Cauchy-Schwarz inequality and \eqref{rvK}, we have
\begin{equation*}
\begin{split}
\int_\Omega \chi(v)\nabla v\cdot \nabla u&\leq \frac{1}{2}\int_\Omega \gamma (v)\frac{|\nabla u|^2}{u}+\frac{1}{2}\int_\Omega \frac{|\chi(v)|^2}{\gamma (v)}|\nabla v|^2 u\\
&\leq \frac{1}{2}\int_\Omega \gamma (v)\frac{|\nabla u|^2}{u}+\frac{K}{2}\|\nabla v\|_{L^4}^2\|u\|_{L^2},
\end{split}
\end{equation*}
which, substituted into \eqref{L2-2}, gives
\begin{equation}\label{L2-3}
\begin{split}
&\frac{d}{dt}\left(\int_\Omega u\ln u-\int_\Omega  u\right)+\frac{1}{2}\int_\Omega \gamma (v)\frac{|\nabla u|^2}{u}\\
&\leq\frac{K}{2}\|\nabla v\|_{L^4}^2\|u\|_{L^2}+\mu\int_\Omega u\ln u-\mu \int_\Omega u^2 \ln u.
\end{split}
\end{equation}
Applying the Agmon-Douglis-Nirenberg  $L^p$ estimates (cf. \cite{A-CPAM-1,A-CPAM-2}) to the second equation of \eqref{1-1*} with homogeneous Neumann boundary conditions, {\color{black} we know that for all $p>1$, there exists a  constant $c_1>0$ such that}
\begin{equation}\label{W2P}
\|v(\cdot,t)\|_{W^{2,p}}\leq c_1\|u(\cdot,t)\|_{L^p}.
\end{equation}
The Sobolev embedding  theorem yields  $\|\nabla v\|_{L^4}\leq c_2\|v\|_{W^{2,\frac{4}{3}}}$ in two dimensions (i.e. $n=2$), which together with \eqref{W2P} implies
\begin{equation}\label{L2-4}
\|\nabla v\|_{L^4}^2\leq c_2^2\|v\|_{W^{2,\frac{4}{3}}}^2\leq c_3\|u\|_{L^\frac{4}{3}}^2.
\end{equation}
On the other hand, using the $L^p$-interpolation inequality and the fact $\|u(\cdot,t)\|_{L^1}\leq m_*$ (see Lemma \ref{L1u}), we have
\begin{equation}\label{L2-5}
\|u\|_{L^\frac{4}{3}}^2\leq \|u\|_{L^2}\|u\|_{L^1}\leq m_*\|u\|_{L^2}.
\end{equation}
We substitute \eqref{L2-4} and \eqref{L2-5} into \eqref{L2-3} to obtain
\begin{equation}\label{L2-6}
\begin{split}
&\frac{d}{dt}\left(\int_\Omega u\ln u-\int_\Omega  u\right)+\frac{1}{2}\int_\Omega \gamma (v)\frac{|\nabla u|^2}{u}+\left(\int_\Omega u\ln u-\int_\Omega  u\right)\\
&\leq\frac{Kc_3m_*}{2}\|u\|_{L^2}^2+(\mu+1)\int_\Omega u\ln u-\mu \int_\Omega u^2 \ln u-\int_\Omega u\\
&\leq\frac{Kc_3m_*}{2}\|u\|_{L^2}^2+(\mu+1)\int_\Omega u\ln u-\mu \int_\Omega u^2 \ln u\\
&\leq c_4,
\end{split}
\end{equation}
where we have used  the facts (see \cite[Lemma 3.1]{TW-JDE-2014}): Let  $\mu>0$ and $A\geq 0$, then there exists a constant $L:=L(\mu, A)>0$ such that
\begin{equation*}
(1+\mu) z\ln z +A z^2-\mu z^2\ln z\leq L, \ \  \mathrm{for\ all }\ \ z>0.
\end{equation*}
Hence from \eqref{L2-6}, we obtain
\begin{equation*}
\frac{d}{dt}\left(\int_\Omega u\ln u-\int_\Omega  u\right)+\int_\Omega u\ln u-\int_\Omega  u\leq c_5,
\end{equation*}
which gives  $\int_\Omega u\ln u-\int_\Omega u \leq c_6$ and then
\begin{equation}\label{L2-7}
\int_\Omega u\ln u\leq c_6+\int_\Omega u\leq c_7.
\end{equation}
Since $u\ln u\geq -\frac{1}{e}$, from \eqref{L2-7} we derive
\begin{equation*}
\int_\Omega |u\ln u|\leq \int_\Omega u\ln u+\frac{2|\Omega|}{e}\leq c_8,
\end{equation*}
which yields \eqref{L2-1}. Finally \eqref{L2-1*} is a consequence of \cite[Lemma A.4]{TW-JDE-2014}) {\color{black}applied to} the second equation of \eqref{1-1*}.
\end{proof}
Next, we will show that there exists some $p>1$ close to $1$ such that $\int_\Omega u^p$ is uniformly bounded in time.
\begin{lemma}\label{L3}
Suppose the conditions in Theorem \ref{GB} hold. Then there exists $p>1$ close to  1 such that
\begin{equation}\label{L3-1}
\|u(\cdot,t)\|_{L^p}\leq C,\ \mathrm{for  \  all}\ \ t\in(0,T_{max}),
\end{equation}
where $C>0$ is a constant independent of $t$.
\end{lemma}
\begin{proof}
We multiply the first equation of \eqref{1-1*} by $u^{p-1}$ to obtain
\begin{equation}\label{L3-2}
\begin{split}
&\frac{1}{p}\frac{d}{dt}\int_\Omega u^p+(p-1)\int_\Omega \gamma (v) u^{p-2}|\nabla u|^2\\
&=(p-1)\int_\Omega \chi(v) u^{p-1}\nabla u\cdot \nabla v+\mu \int_\Omega u^p-\mu \int_\Omega u^{p+1}.
\end{split}
\end{equation}
The Cauchy-Schwarz inequality and \eqref{rvK} allow us to have
\begin{equation}\label{L3-3}
\begin{split}
&(p-1)\int_\Omega \chi(v) u^{p-1}\nabla u\cdot \nabla v\\
&\leq \frac{p-1}{2}\int_\Omega \gamma (v) u^{p-2}|\nabla u|^2+\frac{p-1}{2}\int_\Omega \frac{|\chi(v)|^2}{\gamma (v)} u^p|\nabla v|^2\\
&\leq\frac{ p-1}{2}\int_\Omega \gamma (v) u^{p-2}|\nabla u|^2+\frac{(p-1)K}{2}\int_\Omega u^p|\nabla v|^2.\\
\end{split}
\end{equation}
Using Gagliardo-Nirenberg inequality,  \eqref{L2-1*} and \eqref{W2P}, one has
\begin{equation}\label{L3-4}
\begin{split}
\int_\Omega u^p|\nabla v|^2\leq \|u\|_{L^{p+1}}^p\|\nabla v\|_{L^{2(p+1)}}^2
&\leq c_1 \|u\|_{L^{p+1}}^p\|v\|_{W^{2,p+1}}\|\nabla v\|_{L^2}\\&\leq c_2\|u\|_{L^{p+1}}^p\|v\|_{W^{2,p+1}}
\leq c_3\|u\|_{L^{p+1}}^{p+1}.
\end{split}
\end{equation}
Then we can  substitute \eqref{L3-3} and \eqref{L3-4} into \eqref{L3-2} to obtain
\begin{equation}\label{L3-5}
\frac{1}{p}\frac{d}{dt}\int_\Omega u^p+\frac{(p-1)}{2}\int_\Omega \gamma (v) u^{p-2}|\nabla u|^2\leq \frac{(p-1)Kc_3}{2}\int_\Omega u^{p+1}+\mu \int_\Omega u^p-\mu \int_\Omega u^{p+1}.
\end{equation}
Using  the H\"{o}lder inequality and Cauchy-Schwarz inequality, one can show  that
\begin{equation}\label{L3-6}
(1+\mu)\int_\Omega u^p\leq (1+\mu)|\Omega |^\frac{1}{p+1}\left(\int_\Omega u^{p+1}\right)^\frac{p}{p+1}\leq \frac{\mu}{2}\int_\Omega u^{p+1}+c_4.
\end{equation}
Moreover, we  can choose $p=1+\epsilon>1$ satisfying $\frac{\epsilon K c_3 }{2}<\frac{\mu}{2}$ to derive that
\begin{equation}\label{L3-7}
 \frac{(p-1)Kc_3}{2}\int_\Omega u^{p+1}\leq \frac{\mu}{2}\int_\Omega u^{p+1}.
\end{equation}
Then the combination of \eqref{L3-6}, \eqref{L3-7} and \eqref{L3-5} gives
\begin{equation}\label{L3-8}
\frac{1}{p}\frac{d}{dt}\int_\Omega u^p+\int_\Omega u^p\leq c_4.
\end{equation}
Applying Gronwall's inequality to \eqref{L3-8}, we have \eqref{L3-1} for some $p>1$ close to 1.
\end{proof}
Next, we will show $\|v(\cdot,t)\|_{L^\infty}$ is uniformly bounded in time, which rules out the possibility of  degeneracy.
\begin{lemma}\label{L4}
Suppose the conditions in Theorem \ref{GB} hold. Then there exists a constant $K_1>0$ such that
\begin{equation}\label{L4-1}
\|v(\cdot,t)\|_{L^\infty}\leq K_1, \ \mathrm{for\ all}\ \ t\in(0,T_{max})
\end{equation}
and
\begin{equation}\label{L4-2}
0<\gamma_1\leq\gamma(v)\leq \gamma_2.
\end{equation}
\end{lemma}
\begin{proof}
From Lemma \ref{L3}, we can find  a constant $c_1>0$ such that $\|u(\cdot,t)\|_{L^p}\leq c_1$ for some $p>1$. Then applying the elliptic regularity estimate to the second equation of \eqref{1-1*}, one has $\|v(\cdot,t)\|_{W^{2,p}}\leq c_2\|u(\cdot,t)\|_{L^p}\leq c_1c_2$, which along with the Sobolev inequality give \eqref{L4-1}. Then since $0<\gamma(v)\in C^2([0,\infty))$, we can find two positive constants $\gamma_1$ and $\gamma_2$ such that  \eqref{L4-2} holds.
\end{proof}
\begin{lemma} \label{L5}
Suppose the conditions in Theorem \ref{GB} hold. Then there exists a constant $C>0$ such that
\begin{equation}\label{L5-1}
\|u(\cdot,t)\|_{L^2}\leq C, \ \mathrm{for\ all}\ \ t\in(0,T_{max}).
\end{equation}
\end{lemma}
\begin{proof}
Multiplying the first equation of  \eqref{1-1*} by $u$ and integrating the result by parts,  using Cauchy-Schwarz inequality and  \eqref{rvK}, we end up with
\begin{equation*}\label{L5-2}
\begin{split}
\frac{1}{2}\frac{d}{dt}\int_\Omega u^2+\int_\Omega \gamma(v) |\nabla u|^2+\mu \int_\Omega u^3
&= {\color{black}\int_\Omega \chi(v)u\nabla u\cdot \nabla v}+\mu \int_\Omega u^2\\
&=\frac{1}{2}\int_\Omega \gamma (v) |\nabla u|^2+\frac{1}{2}\int_\Omega \frac{|\chi(v)|^2}{\gamma(v)} u^2|\nabla v|^2+\mu \int_\Omega u^2\\
&\leq \frac{1}{2}\int_\Omega \gamma (v) |\nabla u|^2+\frac{K}{2}\int_\Omega u^2|\nabla v|^2+\mu \int_\Omega u^2, \\
\end{split}
\end{equation*}
which, combined with \eqref{L4-2}, gives
\begin{equation}\label{L5-3}
\begin{split}
\frac{d}{dt}\int_\Omega u^2+\gamma_1\int_\Omega |\nabla u|^2 +2\mu \int_\Omega u^3\leq K\int_\Omega u^2|\nabla v|^2+2\mu \int_\Omega u^2.
\end{split}
\end{equation}
We differentiate the second equation of system $\eqref{1-1*}$ and multiply the result by $2\nabla v$ to obtain
\begin{equation}\label{L5-4}
\begin{split}
 0&=2{\color{black}}\nabla v\cdot\nabla \Delta v+2\nabla v\cdot\nabla u-2|\nabla v|^2\\
&=\Delta|\nabla v|^2-2|D^2v|^2+2\nabla v\cdot\nabla u-2|\nabla v|^2,
\end{split}
\end{equation}
where we have used  the identity $\Delta|\nabla v|^2=2\nabla v\cdot\nabla \Delta v+2|D^2v|^2 $.
Then multiplying \eqref{L5-4} by $|\nabla v|^{2}$ and  integrating the results, we have
\begin{equation}\label{L5-5}
\begin{split}
&\int_\Omega |\nabla|\nabla v|^2|^2+2\int_\Omega |\nabla v|^{2}|D^2v|^2+2\int_\Omega |\nabla v|^4\\
&=\int_{\partial\Omega}|\nabla v|^{2}\frac{\partial |\nabla v|^2}{\partial \nu}dS+2\int_\Omega|\nabla v|^{2}\nabla v\cdot\nabla u\\
&=\int_{\partial\Omega}|\nabla v|^{2}\frac{\partial |\nabla v|^2}{\partial \nu}dS -2\int_\Omega u\Delta v|\nabla v|^{2}-2\int_\Omega  u\nabla(|\nabla v|^{2})\cdot\nabla v \\
&\leq \int_{\partial\Omega}|\nabla v|^{2}\frac{\partial |\nabla v|^2}{\partial \nu}dS+2\int_\Omega u\left(|\Delta v||\nabla v|^{2}+ |\nabla|\nabla v|^{2}||\nabla v|\right).
\end{split}
\end{equation}
With the inequality $\frac{\partial|\nabla v|^2}{\partial \nu}\leq 2\lambda|\nabla v|^2$ on $\partial\Omega$ (see \cite[Lemma 4.2]{MS-2014}) and  the following trace inequality \cite[Remark 52.9]{PP-2007} for any $\varepsilon>0$:
\begin{equation*}\label{Trace inequality}
\|\varphi\|_{L^2(\partial\Omega)}\leq \varepsilon \|\nabla {\color{black}\varphi}\|_{L^2(\Omega)}+C_\varepsilon\|{\color{black}\varphi}\|_{L^2(\Omega)},
\end{equation*}
we have
 \begin{equation}\label{L5-6}
\begin{split}
\int_{\partial\Omega}|\nabla v|^{2}\frac{\partial |\nabla v|^2}{\partial \nu}dS
&\leq 2\lambda\||\nabla v|^2\|_{L^2(\partial \Omega)}^2\leq \frac{1}{4}\int_\Omega |\nabla|\nabla v|^2|^2+c_1\||\nabla v|^2\|_{L^2}^2.
%&\leq D\int_\Omega |\nabla|\nabla v|^2|^2+c_4,
\end{split}
\end{equation}
By the Gagliardo-Nirenberg inequality and the fact $\||\nabla v|^2\|_{L^1}=\|\nabla v\|_{L^2}^2\leq c_2$(see Lemma \ref{L2}), we have
\begin{equation}\label{L5-7}
\begin{split}
c_1\||\nabla v|^2\|_{L^2}^2
&\leq c_3\|\nabla |\nabla v|^2\|_{L^2}\||\nabla v|^2\|_{L^1}+c_3\||\nabla v|^2\|_{L^1}^2\leq  \frac{1}{4}\int_\Omega |\nabla|\nabla v|^2|^2+c_4.
\end{split}
\end{equation}
Then the combination of \eqref{L5-7} and (\ref{L5-6}) gives
\begin{equation}\label{L5-8}
\begin{split}
\int_{\partial\Omega}|\nabla v|^{2}\frac{\partial |\nabla v|^2}{\partial \nu}dS\leq \frac{1}{2}\int_\Omega |\nabla|\nabla v|^2|^2+c_4.
\end{split}
\end{equation}
Next, we will estimate the last term on the right of $\eqref{L5-5}$. To this end, we use the  Young's inequality and the facts  $|\Delta v|\leq \sqrt{2}|D^2 v|$ and
$
\nabla|\nabla v|^{2}=2 D^2 v\cdot\nabla v
$ to derive
\begin{equation}\label{L5-9}
\begin{split}
2\int_\Omega u\left(|\Delta v||\nabla v|^{2}+ \Big|\nabla|\nabla v|^{2}\Big||\nabla v|\right)
&\leq 2\sqrt{2}\int_\Omega u|\nabla v|^{2}|D^2 v|+4\int_\Omega u|\nabla v|^{2}|D^2v|\\
&\leq 2(\sqrt{2}+2)\int_\Omega u|\nabla v|^{2}|D^2v|\\
&\leq 2\int_\Omega |\nabla v|^{2}|D^2 v|^2+\frac{(2+\sqrt{2})^2}{2}\int_\Omega u^2|\nabla v|^{2}.
\end{split}
\end{equation}
Substituting $\eqref{L5-8}$ and $\eqref{L5-9}$ into $\eqref{L5-5}$, one has
\begin{equation}\label{L5-10}
\int_\Omega |\nabla|\nabla v|^2|^2+4\int_\Omega |\nabla v|^4\leq (2+\sqrt{2})^2\int_\Omega u^2|\nabla v|^{2}+2c_4.
\end{equation}
Combining \eqref{L5-3} and \eqref{L5-10} and using the  Young's inequality, we can find some $\zeta>0$ such that
\begin{equation}\label{L5-11}
\begin{split}
&\frac{d}{dt}\int_\Omega u^2+\gamma_1\int_\Omega |\nabla u|^2 +2\mu \int_\Omega u^3 +\int_\Omega |\nabla|\nabla v|^2|^2+4\int_\Omega |\nabla v|^4\\
&\leq [K+ (2+\sqrt{2})^2]\int_\Omega u^2|\nabla v|^2+2\mu \int_\Omega u^2+2c_4\\
&\leq  [K+ (2+\sqrt{2})^2]\|u\|_{L^3}^2\|\nabla v\|_{L^6}^2+2\mu |\Omega|^\frac{1}{3}\|u\|_{L^3}^2+{\color{black}2c_4}\\
&\leq c_5\|u\|_{L^3}^3+\zeta\|\nabla v\|_{L^6}^6+\mu \|u\|_{L^3}^3+c_6.
\end{split}
\end{equation}
With the boundedness of  $\|u\|_{L^1}$ and $\|u\ln u\|_{L^1}$ and the inequality in \cite[Lemma 3.5]{Nagai-Funk},   we can choose $\varepsilon$ small enough to obtain
\begin{equation}\label{L5-12}
\begin{split}
\|u\|_{L^3}^3
 &\leq \varepsilon \|\nabla u\|_{L^2}^2\|u\ln u\|_{L^1}+C_\varepsilon(\|u\ln u\|_{L^1}^3+\|u\|_{L^1})\leq \frac{\gamma_1}{c_5}\|\nabla u\|_{L^2}^2+c_7.
\end{split}
\end{equation}
On the other hand, using the Gagliardo-Nirenberg inequality, we can derive that
\begin{equation}\label{L5-13}
\begin{split}
\|\nabla v\|_{L^6}^6=\||\nabla v|^2\|_{L^3}^3
&\leq c_8(\|\nabla |\nabla v|^2\|_{L^2}^2\||\nabla v|^2\|_{L^1}+\||\nabla v|^2\|_{L^1}^3)\\
&\leq c_8c_2\|\nabla |\nabla v|^2\|_{L^2}^2+c_8c_2^3.
\end{split}
\end{equation}
Substituting \eqref{L5-12} and \eqref{L5-13} into \eqref{L5-11}, and choosing $\zeta=\frac{1}{c_2c_8}$, we end up with
$
\frac{d}{dt}\int_\Omega u^2+\mu \int_\Omega u^3\leq  c_{11}
$
which along with the Young inequality: $\int_\Omega u^2 \leq \mu \int_\Omega u^3 +c_{12}$ yields
\begin{equation*}
\frac{d}{dt}\int_\Omega u^2+\int_\Omega u^2\leq  c_{11}+c_{12}.
\end{equation*}
This gives \eqref{L5-1} by Gronwall's inequality.
\end{proof}
Next, we shall show the boundedness of $\|u(\cdot,t)\|_{L^\infty}$. To this end, we first improve the regularity of $v$. More precisely, we have the following results.
\begin{lemma}\label{L6}
Suppose the conditions in Theorem \ref{GB} hold. Then we have
\begin{equation}\label{L6-1}
\|\nabla v(\cdot,t)\|_{L^\infty}\leq C, \ \mathrm{for\ all}\ \ t\in(0,T_{max}),
\end{equation}
where $C>0$ is a constant independent of $t$.
\end{lemma}
\begin{proof}
Using \eqref{W2P} and the fact $\|u(\cdot,t)\|_{L^2}\leq c_1$, we can derive that
$
\|v(\cdot,t)\|_{W^{2,2}}\leq c_2\|u(\cdot,t)\|_{L^2}\leq c_1c_2,
$
which by the Sobolev embedding theorem ($n=2$) gives
\begin{equation}\label{L6-2}
\|\nabla v\|_{L^4}\leq c_3.
\end{equation}
Then multiplying the first equation of  \eqref{1-1*} by $u^2$ and integrating it over $\Omega$ by parts, one obtains
\begin{equation*}
\begin{split}
\frac{1}{3}\frac{d}{dt}\int_\Omega u^3+2\int_\Omega \gamma(v)u|\nabla u|^2+\mu \int_\Omega u^4
&=2\int_\Omega u^2\chi(v)\nabla u\cdot\nabla v+\mu \int_\Omega u^3\\
&\leq \int_\Omega \gamma(v)u|\nabla u|^2+\int_\Omega \frac{|\chi(v)|^2}{\gamma(v)} u^3|\nabla v|^2+\frac{\mu}{2}\int_\Omega u^4+c_4,
\end{split}
\end{equation*}
which subject to the facts \eqref{rvK} and \eqref{L6-2} gives rise to
\begin{equation}\label{L6-3}
\begin{split}
\frac{1}{3}\frac{d}{dt}\int_\Omega u^3+\frac{4\gamma_1}{9}\int_\Omega |\nabla u^\frac{3}{2}|^2+\frac{\mu}{2} \int_\Omega u^4
&\leq K\int_\Omega u^3|\nabla v|^2+c_4\\
&\leq K\|u\|_{L^6}^3\|\nabla v\|_{L^4}^2+c_4\\
&\leq c_3^2K \|u\|_{L^6}^3+c_4.
\end{split}
\end{equation}
Using the Gagliardo-Nirenberg inequality with the fact $\|u^\frac{3}{2}\|_{L^\frac{4}{3}}=\|u\|_{L^2}^\frac{3}{2}\leq c_5$, we can show that
\begin{equation}\label{L6-4}
\begin{split}
c_3^2K \|u\|_{L^6}^3=c_3^2K\|u^\frac{3}{2}\|_{L^4}^2
&\leq c_6\left(\|\nabla u^\frac{3}{2}\|_{L^2}^{\frac{4}{3}}\|u^\frac{3}{2}\|_{L^\frac{4}{3}}^{\frac{2}{3}}+\|u^\frac{3}{2}\|_{L^\frac{4}{3}}^2\right)\\
&\leq c_7\|\nabla u^\frac{3}{2}\|_{L^2}^\frac{4}{3}+c_7\\
&\leq\frac{4\gamma_1}{9}\int_\Omega |\nabla u^\frac{3}{2}|^2+c_8.
\end{split}
\end{equation}
On the other hand, using the H\"{o}lder inequality and Young inequality, one has
\begin{equation}\label{L6-5}
\int_\Omega u^3\leq |\Omega|^\frac{1}{4}\left(\int_\Omega u^4\right)^\frac{3}{4}\leq\frac{\mu}{2}\int_\Omega u^4+c_9.
\end{equation}
Substituting \eqref{L6-4} and \eqref{L6-5} into \eqref{L6-3} gives
\begin{equation*}
\frac{1}{3}\frac{d}{dt}\int_\Omega u^3+\int_\Omega u^3\leq c_{10},
\end{equation*}
which along with the Gronwall's inequality gives
\begin{equation}\label{L6-6}
\|u(\cdot,t)\|_{L^3}\leq {\color{black}c_{11}}.
\end{equation}
Using the elliptic regularity \eqref{W2P} and Sobolev embedding theorem again, from \eqref{L6-6} we  derive
\begin{equation*}\label{L6-7}
\|\nabla v\|_{L^\infty}\leq {\color{black}c_{12}\|v\|_{W^{2,3}}\leq c_{13}\|u\|_{L^3}\leq c_{11}c_{13}.}
\end{equation*}
This finishes the proof.
\end{proof}
\begin{lemma}\label{L7}
Suppose the  conditions in Theorem \ref{GB} hold. Then the solution of  \eqref{1-1*} satisfies
\begin{equation}\label{L7-1}
\|u(\cdot,t)\|_{L^\infty}\leq C, \ \mathrm{for\ all}\ \ t\in(0,T_{max}),
\end{equation}
where the constant $C>0$  independent of $t$.
\end{lemma}
\begin{proof}
Multiplying the first equation of  \eqref{1-1*} by {\color{black}$u^{p-1}(p\geq 2)$ and integrating it by parts over $\Omega$,  and using \eqref{L6-1} and  Young's inequality, we can find a constant $c_1>0$ independent of $p$ such that}
\begin{equation}\label{L7-2}
\begin{split}
&\frac{1}{p}\frac{d}{dt}\int_\Omega u^p +(p-1)\int_\Omega \gamma(v) u^{p-2}|\nabla u|^2+\mu \int_\Omega u^{p+1}\\
&=(p-1)\int_\Omega \chi(v)u^{p-1}\nabla u\cdot\nabla v+\mu\int_\Omega u^p\\
&\leq c_1(p-1)\int_\Omega |\chi(v)|u^{p-1}|\nabla u|+\mu(p-1)\int_\Omega u^p\\
&\leq \frac{p-1}{2}\int_\Omega \gamma(v) u^{p-2}|\nabla u|^2+\left(\frac{c_1^2K}{2}+\mu\right)(p-1)\int_\Omega u^p,
\end{split}
\end{equation}
which, together with the fact $\gamma(v)\geq \gamma_1>0$ in \eqref{L4-2}, gives a positive  constant $c_2=\frac{c_1^2K}{2}+\mu+1$ such that
\begin{equation}\label{L7-5}
\begin{split}
&\frac{d}{dt}\int_\Omega u^p+p(p-1)\int_\Omega u^p+\frac{2(p-1)\gamma_1}{p}\int_\Omega |\nabla u^\frac{p}{2}|^2\\
&\leq c_2p(p-1)\int_\Omega u^p\\
&\leq \frac{2(p-1)\gamma_1}{p}\int_\Omega |\nabla u^\frac{p}{2}|^2+c_3p(p-1)(1+p^2)\left(\int_\Omega u^\frac{p}{2}\right)^2,
\end{split}
\end{equation}
{\color{black}where the last inequality is obtained based on the following inequality (see \cite{TW-M3AS-2013})}
%Using \eqref{L6-1} and  Young's inequality, we can find two positive constants $c_1$ and $c_2$ independent of  $p$ such that
%\begin{equation}\label{L7-3}
%\begin{split}
%(p-1)\int_\Omega \chi(v)u^{p-1}\nabla u\cdot\nabla v
%&\leq c_1(p-1)\int_\Omega |\chi(v)|u^{p-1}|\nabla u|\\
%&\leq \frac{p-1}{2}\int_\Omega \gamma(v) u^{p-2}|\nabla u|^2+\frac{c_1^2K(p-1)}{2}\int_\Omega u^p
%\end{split}
%\end{equation}
%and
%\begin{equation}\label{L7-4}
%\mu\int_\Omega u^p\leq \mu\int_\Omega u^{p+1}+c_2.
%\end{equation}
%Substituting \eqref{L7-3}, \eqref{L7-4} into \eqref{L7-2} and  using the facts $\gamma(v)\geq \gamma_1>0$ in \eqref{L4-2}
\begin{equation*}\label{L7-6}
\|f\|_{L^2}^2\leq \varepsilon \|\nabla f\|_{L^2}^2+ c_4(1+\varepsilon^{-1}) \|f\|_{L^1}^2,\ \ \ \mathrm{for\ any\ }\ \ \varepsilon>0.
\end{equation*}
%one can find a positive constant $c_3=\frac{c_1^2K}{2}+1$ such that
%\begin{equation*}\label{L7-5}
%\begin{split}
%&\frac{d}{dt}\int_\Omega u^p+p(p-1)\int_\Omega u^p+\frac{2(p-1)\gamma_1}{p}\int_\Omega |\nabla u^\frac{p}{2}|^2\\
%&\leq c_3p(p-1)\int_\Omega u^p+c_2 p\\
%&\leq \frac{2(p-1)\gamma_2}{p}\int_\Omega |\nabla u^\frac{p}{2}|^2+c_4p(p-1)(1+p^n)\left(\int_\Omega u^\frac{p}{2}\right)^2+c_2p,
%\end{split}
%\end{equation*}
The inequality \eqref{L7-5} can be rewritten as
\begin{equation*}
\frac{d}{dt}\int_\Omega u^p+p(p-1)\int_\Omega u^p\leq c_3p(p-1)(1+p^2)\left(\int_\Omega u^\frac{p}{2}\right)^2,
\end{equation*}
which, combined with  the fact $(1+p^2)\leq (1+p)^2$, gives
\begin{equation}\label{L7-7}
\frac{d}{dt}\left( e^{p(p-1)t}\int_\Omega u^p\right)\leq c_3 e^{p(p-1)t}p(p-1)(1+p)^2 \left(\int_\Omega u^\frac{p}{2}\right)^2.
\end{equation}
We integrate \eqref{L7-7} over $[0,t]$ for $0<t<T_{max}$ to obtain
\begin{equation}\label{L7-8}
\int_\Omega u^p\leq \int_\Omega u_0^p+c_3(1+p)^2\sup\limits_{0\leq t\leq T_{max}}\left(\int_\Omega u^\frac{p}{2}\right)^2.
\end{equation}
Define
\begin{equation}\label{L7-9}
N(p):=\max \Big\{\|u_0\|_{L^\infty},\sup\limits_{0\leq t\leq T_{max}}\left(\int_\Omega u^p\right)^\frac{1}{p}\Big\}.
\end{equation}
Then, we can derive from \eqref{L7-8} and \eqref{L7-9} that
\begin{equation*}
N(p)\leq [c_4(1+p)^2]^\frac{1}{p}N(\frac{p}{2}) \ \ \mathrm{for}\ \  p\geq 2.
\end{equation*}
Taking $p=2^j, j=1,2,\cdots,$ one obtains
\begin{equation*}
\begin{split}
N(2^j)&\leq c_4^{2^{-j}}(1+2^j)^{2^{-j+1} }N(2^{j-1})\\
          &\vdots\\
          &\leq c_4^{2^{-j}+\cdots+2^{-1}}(1+2^j)^{2^{-j+1}}\cdots(1+2)N(1)\\
          &\leq  c_4[2^{j2^{-j+1}}(2^{-j}+1)^{2^{-j+1}}]\cdots[2(2^{-1}+1)]N(1)\\
          &\leq c_4 2^{2[j2^{-j}+(j-1)2^{-(j-1)}+\cdots+2^{-1}]}\cdot 2^{2[2^{-j}+2^{-(j-1)}+\cdots+2^{-1}]}N(1)\\
          &\leq c_4 2^{6} N(1).
 \end{split}
\end{equation*}
Letting $j\to\infty$ and noting the boundedness of $\|u\|_{L^1}$, we have
\begin{equation*}
\|u(\cdot,t)\|_{L^\infty}\leq c_4 2^{6} N(1)\leq c_4 2^{6} \max \{\|u_0\|_{L^\infty},\|u_0\|_{L^1}\}\leq c_5,
\end{equation*}
which gives \eqref{L7-1}.
\end{proof}
\begin{lemma}\label{L8}
Let $\Omega$ be a bounded domain  in $\R^2$ with smooth boundary and the hypothesis (H1) hold. Suppose that $u_0\in W^{1,\infty}(\Omega)$ with $u_0\geq 0(\not \equiv 0)$. Then  the problem \eqref{1-1*} has a unique solution $[C^0([0,\infty)\times \bar{\Omega})\cap C^{2,1}((0,\infty)\times\bar{\Omega})]\times C^{2,1}((0,\infty)\times\bar{\Omega}) $, which satisfies
 \begin{equation*}\label{L1-1*}
 \|u(\cdot,t)\|_{L^\infty}+\|v(\cdot,t)\|_{W^{1,\infty}}\leq C.
 \end{equation*}
\end{lemma}
\begin{proof} From Lemma \ref{L7}, we can find a constant $c_1>0$ such that $\|u(\cdot,t)\|_{L^\infty}\leq c_1$. Then using the elliptic regularity, from the second equation of \eqref{1-1*} one obtains $\|v(\cdot,t)\|_{W^{1,\infty}}\leq c_2$. By Lemma \ref{LS}, the existence of global classical solutions follows immediately.
\end{proof}
\subsection{Large time behavior}
In this section, we will study the large time behavior of solution for the system \eqref{1-1*}. Let
\begin{equation}\label{K0}
K_0=\max\limits_{0\leq v\leq \infty}\frac{|\chi(v)|^2}{\gamma(v)}
\end{equation}
and
\begin{equation}\label{LT**}
\begin{split}
\mathcal{E}(t):=\int_\Omega\left(u-1-\ln u\right).
\end{split}
\end{equation}
Then {\color{black}based on} some ideas in \cite{JKW-SIAP-2018,Tao-Winkler-SIMA-2015}, we shall show that  the constant steady state $(1,1)$ is globally asymptotically stable by showing $\mathcal{E}(t)$ is a Lyapunov functional under the conditions $\mu>\frac{K_0}{16}$. More precisely, we have the following result.
\begin{lemma}\label{LT}
Suppose $(u,v)$ is the solution of \eqref{1-1*} obtained in Lemma \ref{L8}. Let $K_0$ and $\mathcal{E}(t)$ be defined by \eqref{K0} and \eqref{LT**}, respectively. Then we have the following results:
 \begin{itemize}
\item[(1)] $\mathcal{E}(t)\geq 0$ for any $t>0$;
\item[(2)] If $\mu>\frac{K_0}{16}$,
 then there exists a positive constant $\beta$ such that  for all $t>0$
\begin{equation}\label{LT*}
\mathcal{E}'(t)\leq -\mathcal{F}(t),
\end{equation}
where
\begin{equation*}
\mathcal{F}(t):=\beta\cdot\Big\{\int_\Omega (u-1)^2+\int_\Omega (v-1)^2\Big\}.
\end{equation*}
\end{itemize}
\end{lemma}
\begin{proof}
First, we will show the non-negativity of $\mathcal{E}(t)$. In fact, letting  $\phi(u):=u-1-\ln u,u>0$ and noting that $\phi(1)=\phi'(1)=0$, and applying the Taylor's formula to $\phi(u)$ at $u=1$ gives
\begin{equation}\label{LT-6}
\phi(u)=\frac{1}{2}\phi''(\tilde{u})(u-1)^2=\frac{1}{2\tilde{u}^2}(u-1)^2\geq 0,
\end{equation}
where $\tilde{u}$ is between $1$ and $u$, which implies $\mathcal{E}(t)\geq 0$.

 Next, we show \eqref{LT*} hold. In fact, using the first equation of \eqref{1-1*}, we have
\begin{equation}\label{LT-1}
\begin{split}
\mathcal{E}'(t)=\frac{d}{dt}\int_\Omega (u-1-\ln u)
&=-\int_\Omega \nabla \left(\frac{u-1}{u}\right) \cdot[\gamma(v)\nabla u-{\color{black}\chi(v) u\nabla v}]-\mu \int_\Omega (u-1)^2\\
&=-\int_\Omega \gamma(v)\frac{|\nabla u|^2}{u^2}+\int_\Omega \chi(v)\frac{\nabla u\cdot\nabla v}{u}-\mu \int_\Omega (u-1)^2.
\end{split}
\end{equation}
On the other hand, we multiply the second equation of system \eqref{1-1*} by $v-1$ and integrate it by parts to obtain
\begin{equation}\label{LT-2}
0=-\int_\Omega |\nabla v|^2-\int_\Omega (v-1)^2+\int_\Omega (u-1)(v-1).
\end{equation}
Multiplying \eqref{LT-2} by a constant $\delta>0$ and adding the result to \eqref{LT-1}, we obtain
\begin{equation}\label{LT-3}
\begin{split}
\frac{d}{dt}\int_\Omega (u-1-\ln u)
&=\underbrace{ -\int_\Omega \gamma(v)\frac{|\nabla u|^2}{u^2}-\delta\int_\Omega |\nabla v|^2+\int_\Omega \chi(v)\frac{\nabla u\cdot\nabla v}{u}}_{I_1}\\
&\ \ \ \ \underbrace{-\mu \int_\Omega (u-1)^2-\delta\int_\Omega (v-1)^2+\delta\int_\Omega (u-1)(v-1)}_{I_2}.\\
\end{split}
\end{equation}
For $I_1$, we can rewrite it as
$$I_1=-\Theta_1^T A_1 \Theta_1,\ \Theta_1=\left(\begin{array}{c}\nabla u\\[1mm]\nabla v \end{array}\right), \ \ A_1=\left(\begin{array}{cc}\frac{\gamma(v)}{u^2} & -\frac{\chi(v)}{2 u} \\[1mm]-\frac{\chi(v)}{2 u} &\delta \end{array}\right)$$
where $\Theta_1^T$ denotes the transpose of $\Theta_1$. One can check that  $A_1$ is non-negative definite if and only if
\begin{equation}\label{LT-4}
\delta \geq \max\limits_{0\leq v\leq \infty}\frac{|\chi(v)|^2}{4\gamma(v)}=\frac{K_0}{4}.
\end{equation}
Similarly, we can also rewrite $I_2$ as
$$I_2=-\Theta_2^T A_2 \Theta_2,\ \Theta_2=\left(\begin{array}{c} u-1\\[1mm] v-1 \end{array}\right), \ \ A_2=\left(\begin{array}{cc}\mu & \frac{\delta}{2} \\[1mm]\frac{\delta}{2 } &\delta \end{array}\right).$$
$A_2$ is positive definite if and only if
\begin{equation}\label{LT-5}
\mu>\frac{\delta}{4}.
\end{equation}
Hence, we can always find  a positive constant $\delta$ such that \eqref{LT-4} and\eqref{LT-5} hold provided $\mu>\frac{K_0}{16}$. Since $A_1$ is non-negative definite and $A_2$ is positive definite,
then from \eqref{LT-3}, we can find a constant  $\beta>0$ such that \eqref{LT*} holds.

\end{proof}

Next, we will use \eqref{LT*} to show the convergence of solution $(u,v,w)$ in  $L^\infty$-norm. Before that, we first improve the regularity of solutions $(u,v)$.
\begin{lemma}\label{L9}
There exist $\sigma\in(0,1)$ and $C>0$ such that
\begin{equation}\label{L9-1}
\|u\|_{{C^{\sigma,\frac{\sigma}{2}}}(\bar{\Omega}\times[t,t+1])}\leq C, \ \mathrm{for}\ all\ \  t\geq 0.
\end{equation}
\end{lemma}
\begin{proof}
From Lemma \ref{L8}, we can find three positive constants $c_1,c_2,c_3$ such that
\begin{equation*}
0<u(x,t)\leq c_1, 0<v(x,t)\leq c_2 \  \mathrm{and}\ \ |\nabla v(x,t)|\leq c_3 \ \mathrm{for \ all }\ x\in \Omega \ \mathrm{and} \ \  t\in (0,T_{max}).
\end{equation*}
The first equation of  \eqref{1-1*} can be rewritten as
\begin{equation}\label{L9-3}
u_t=\nabla \cdot A(x,t,\nabla u)+B(x,t) \ \mathrm{for\ all}\  x\in \Omega \ \mathrm{and} \ \ t\in (0,T_{max}),
\end{equation}
where
\begin{equation*}
A(x,t,\xi):=\gamma(v)\cdot \xi-\chi(v)u \nabla v
\end{equation*}
and
\begin{equation*}
B(x,t):=\mu u(\cdot,t)(1-u(\cdot,t)).
\end{equation*}
Noting the assumptions in (H1) and using the Young's inequality, we can obtain that
\begin{equation}\label{L9-4}
\begin{split}
A(x,t,\nabla u)\cdot \nabla u
&=\gamma(v)|\nabla u|^2-\chi(v)u\nabla v\cdot\nabla u\\
&\geq \gamma(v)|\nabla u|^2-|\chi(v)| u|\nabla v| |\nabla u|\\
&\geq \frac{\gamma(v)}{2}|\nabla u|^2-\frac{|\chi(v)|^2}{2\gamma(v)} u^2|\nabla v|^2\\
\end{split}
\end{equation}
and
\begin{equation*}
|A(x,t,\nabla u)|\leq \gamma_2|\nabla u|+c_4 \ \mathrm{for\ all}\  x\in \Omega \ \mathrm{and} \ \ t\in (0,T_{max})
\end{equation*}
as well as
\begin{equation}\label{L9-6}
|B(x,t)|\leq \mu c_1(1+c_1) \ \mathrm{for\ all}\  x\in \Omega \ \mathrm{and} \ \ t\in (0,T_{max}).
\end{equation}
Then \eqref{L9-4}-\eqref{L9-6} allow us to apply the H\"{o}lder regularity for quasilinear parabolic equations \cite[Theorem 1.3 and Remark 1.4]{RT} to conclude that $u$ satisfies \eqref{L9-1}.
\end{proof}

\begin{lemma}\label{LT3}
Suppose that $\mu>\frac{K_0}{16}$ and let $(u,v)$ be the global classical solution of the system \eqref{1-1*}. Then it follows that
\begin{equation}\label{LT3-1}
\|u(\cdot,t)-1\|_{L^\infty}\to 0,\ \  \mathrm{as} \ \ \ t\to\infty
\end{equation}
and
\begin{equation}\label{LT3-2}
 \|v(\cdot,t)-1\|_{L^\infty}\to 0, \ \ \mathrm{as} \ \ \ t\to\infty.
\end{equation}
\end{lemma}
\begin{proof}
From Lemma \ref{LT}, we know $\mathcal {E}(t)\geq 0$ for all $t>0$. Then integrating  \eqref{LT*} over $[1,t]$,  we have
\begin{equation*}
\int_1^t\mathcal{F}(s)ds\leq \mathcal{E}(1)-\mathcal{E}(t)\leq \mathcal{E}(1),\ \ \  \mathrm{for\ all} \ t>1.
\end{equation*}
Using the definition of $\mathcal{F}(t)$, one can derive
\begin{equation}\label{LT2-1}
\int_1^t\int_\Omega \left[(u-1)^2+ \left(v-1\right)^2
\right]<\infty.
\end{equation}
Then combining \eqref{LT2-1} and Lemma \ref{L9}, and using a similar argument  as in \cite[Lemma 4.2]{JKW-SIAP-2018}, we obtain \eqref{LT3-1}. On the other hand, from the second equation of \eqref{1-1*}, we infer that $\psi(x,t):=v(x,t)-1$ satisfies
\begin{equation}\label{ep}
\begin{cases}
-\Delta \psi+\psi=u-1, &x\in\Omega, t>0,\\
\frac{\partial\psi}{\partial \nu}=0,&x\in\Omega, t>0.
\end{cases}
\end{equation}
Then using the {\color{black}elliptic  maximum principle}, we obtain from \eqref{ep} that
\begin{equation}\label{vc}
\|v(\cdot,t)-1\|_{L^\infty}=\|\psi(\cdot,t)\|_{L^\infty}\leq \|u(\cdot,t)-1\|_{L^\infty},
\end{equation}
which together with \eqref{LT3-1} gives \eqref{LT3-2}.
\end{proof}
\subsection{Exponential decay} Next, we shall show the convergence rate is exponential.
\begin{lemma}\label{LT5}
Assume that $\mu>\frac{K_0}{16}$, and suppose $(u,v)$ is the global classical solution of the system \eqref{1-1*}.  Then there exists two positive constants $C,\delta_*$ such that for all $t>0$
\begin{equation}\label{LT5-1}
\|u(\cdot,t)-1\|_{L^2}\leq C e^{-\frac{\delta_*}{2} t}.
\end{equation}
\end{lemma}
\begin{proof} From \eqref{LT3-1}, we can get a $t_0>0$ such that  for all $t> t_0$
\begin{equation*}
\|u(\cdot,t)-1\|_{L^\infty}<\frac{1}{2},
\end{equation*}
which immediately gives
\begin{equation}\label{LT4-1}
u(x,t)\in \left(\frac{1}{2},\frac{3}{2}\right) \ \ \mathrm{for \ all}\ \  x\in\Omega \ \ \mathrm{and}\ \ t>t_0.
\end{equation}
Then using \eqref{LT-6} and \eqref{LT4-1}, we can get two positive constants $c_1$ and $c_2$ such that
\begin{equation}\label{LT4-3}
c_1(u-1)^2\leq u-1-\ln u\leq c_2(u-1)^2 \ \  \mathrm{for \  all}\ \  u\in\left(\frac{1}{2},\frac{3}{2}\right).
\end{equation}
Hence, using \eqref{LT**} and \eqref{LT4-3}, and choosing  $\delta_*=\frac{\beta}{c_2}$, we have for all $t>t_0$ that
\begin{equation*}\label{T4-9-4}
\begin{split}
\mathcal{E}(t)
&\leq c_2\int_\Omega (u-1)^2\leq \frac{1}{\delta_*}\mathcal{F}(t),
\end{split}
\end{equation*}
which yields
\begin{equation}\label{LT4-2}
\mathcal{F}(t)\geq \delta_* \mathcal{E}(t)\ \ \ \ \ \ \mathrm{for\ all \ }t>t_0.
\end{equation}
Then the combination of \eqref{LT*} and \eqref{LT4-2} gives for all $t>t_0$
\begin{equation*}\label{LT5-3}
\mathcal{E}'(t)\leq -\mathcal{F}(t)\leq -\delta_* \mathcal{E}(t),
\end{equation*}
and hence
\begin{equation*}\label{LT5-4}
\mathcal{E}(t)\leq \mathcal{E}(t_0)e^{-\delta_*(t-t_0)}, \ \ \mathrm{for\ all}\ \  t>t_0,
\end{equation*}
which together with the fact
$
\mathcal{E}(t)\geq c_1\int_\Omega(u-1)^2
$
gives \eqref{LT5-1}. Then we finish the proof of Lemma \ref{LT5}.
\end{proof}
%
%
%
%
%
%\begin{lemma}\label{LT4}
%Let $\mu>\frac{K_0}{16}$ and $(u,v)$ be the global classical solution of the system \eqref{1-1*}. Then there exist $t_0>0$ and $\delta_*>0$ such that
%\begin{equation}\label{LT4-1}
%u(x,t)\in \left(\frac{1}{2},\frac{3}{2}\right) \ \ \mathrm{for \ all}\ \  x\in\Omega \ \ \mathrm{and}\ \ t>t_0,
%\end{equation}
%and  such that the functions $\mathcal{E}(t)$ and $\mathcal{F}(t)$ defined in Lemma \ref{LT} satisfy
%\begin{equation}\label{LT4-2}
%\mathcal{F}(t)\geq \delta_* \mathcal{E}(t)\ \ \ \ \ \ \mathrm{for\ all \ }t>t_0.
%\end{equation}
%\end{lemma}

%\begin{lemma}\label{LT5}
%Assume that $\mu>\frac{K_0}{16}$, and suppose $(u,v)$ is the global classical solution of the system \eqref{1-1*}.  Then there exists a constant $C>0$ such that for all $t>0$
%\begin{equation}\label{LT5-1}
%\|u(\cdot,t)-1\|_{L^2}\leq C e^{-\frac{\delta_*}{2} t},
%\end{equation}
%where $\delta_*$ is given in Lemma \ref{LT4}.
%\end{lemma}
%\begin{proof} The combination of  Lemma \ref{LT4} and \eqref{LT*} show that there exists $t_0>0$ such that for all $t>t_0$, one has $u(x,t)\in(\frac{1}{2},\frac{3}{2})$ and
%\begin{equation}\label{LT5-3}
%\mathcal{E}'(t)\leq -\mathcal{F}(t)\leq -\delta_* \mathcal{E}(t).
%\end{equation}
%After some calculations, from \eqref{LT5-3}, we have
%\begin{equation*}\label{LT5-4}
%\mathcal{E}(t)\leq \mathcal{E}(t_0)e^{-\delta_*(t-t_0)}, \ \ \mathrm{for\ all}\ \  t>t_0,
%\end{equation*}
%which together with the fact
%$
%\mathcal{E}(t)\geq c_1\int_\Omega(u-1)^2
%$
%gives \eqref{LT5-1}. Then we finish the proof of Lemma \ref{LT5}.
%\end{proof}
Next, we shall show the boundedness of $\|\nabla u\|_{L^4}$ to obtain the convergence rate with $L^\infty$-norm. More precisely, we have the following results.
\begin{lemma}\label{L4e}
There exists a constant $C>0$ independent of $t$ such that the solution $(u,v)$ of \eqref{1-1*} satisfies
\begin{equation}\label{L4e-1}
\|\nabla u(\cdot,t )\|_{L^4}\leq C \ \ \mathrm{for\ all}\ \ t\in(0,T_{max}).
\end{equation}
\end{lemma}
\begin{proof} Using the first equation of \eqref{1-1*}, we obtain
\begin{equation}\label{L4e-2}
\begin{split}
\frac{1}{4}\frac{d}{dt}\int_\Omega |\nabla u|^4
=&\int_\Omega |\nabla u|^2\nabla u\cdot \nabla u_t\\
=&\int_\Omega |\nabla u|^2\nabla u\cdot \nabla (\nabla \cdot(\gamma(v)\nabla u))-\int_\Omega |\nabla u|^2\nabla u\cdot \nabla (\nabla \cdot(\chi(v)u\nabla v))\\
&+\mu \int_\Omega (1-2u)|\nabla u|^4\\
=&: J_1+J_2+J_3.
\end{split}
\end{equation}
We can estimate the term $J_1$ as follows:
\begin{equation}\label{J1}
\begin{split}
J_1
&=-\int_\Omega |\nabla u|^2\Delta u  \nabla \cdot(\gamma(v)\nabla u)-\int_\Omega \nabla |\nabla u|^2\cdot \nabla u\nabla \cdot(\gamma(v)\nabla u)\\
&=\int_\Omega \gamma(v)|\nabla u|^2\nabla \Delta u\cdot \nabla u-\int_\Omega \gamma'(v)\nabla |\nabla u|^2\cdot \nabla u\nabla u\cdot\nabla v\\
&=\frac{1}{2}\int_\Omega\gamma(v)|\nabla u|^2 \Delta |\nabla u|^2-\int_\Omega \gamma(v) |\nabla u|^2 |D^2 u|^2-\int_\Omega \gamma'(v)\nabla |\nabla u|^2\cdot \nabla u\nabla u\cdot\nabla v\\
&=\frac{1}{2}\int_{\partial\Omega} \gamma(v)|\nabla u|^2\frac{\partial |\nabla u|^2}{\partial \nu}dS-\frac{1}{2}\int_\Omega \gamma'(v)|\nabla u|^2\nabla v\cdot \nabla |\nabla u|^2-\frac{1}{2}\int_\Omega \gamma(v)|\nabla |\nabla u|^2|^2\\
&\ \ \ -\int_\Omega \gamma(v) |\nabla u|^2 |D^2 u|^2-\int_\Omega \gamma'(v)\nabla |\nabla u|^2\cdot \nabla u\nabla u\cdot\nabla v\\
&\leq \frac{1}{2}\int_{\partial\Omega} \gamma(v)|\nabla u|^2\frac{\partial |\nabla u|^2}{\partial \nu}dS-\frac{1}{2}\int_\Omega \gamma(v)|\nabla |\nabla u|^2|^2 -\int_\Omega \gamma(v) |\nabla u|^2 |D^2 u|^2\\
&\ \ \ \ +\frac{3}{2}\int_\Omega |\gamma'(v)||\nabla |\nabla u|^2||\nabla u|^2|\nabla v|.
\end{split}
\end{equation}
Using the boundedness of $\|u\|_{L^\infty}$ and $\|v\|_{W^{1,\infty}}$ obtained in Lemma \ref{L8} and the assumptions in (H1) as well as the fact $\Delta v=v-u$, we have
\begin{equation*}
\begin{split}
\nabla \cdot (\chi(v)u\nabla v)
&=\chi'(v) u |\nabla v|^2+\chi(v)\nabla u\cdot\nabla v+\chi(v) u\Delta v\\
&=\gamma''(v) u |\nabla v|^2+\chi(v)\nabla u\nabla v+\chi(v) uv-\chi(v)u^2\\
&\leq c_1(1+|\nabla u|),
\end{split}
\end{equation*}
which substituted into $J_2$ gives
\begin{equation}\label{J2}
\begin{split}
J_2
&=\int_\Omega \nabla |\nabla u|^2\cdot \nabla u \nabla \cdot (\chi(v)u\nabla v)+\int_\Omega |\nabla u|^2\Delta u\nabla \cdot (\chi(v)u\nabla v)\\
&\leq c_1\int_\Omega |\nabla u||\nabla |\nabla u|^2|(1+|\nabla u|)+c_1\int_\Omega |\nabla u|^2|\Delta u|(1+|\nabla u|).\\
\end{split}
\end{equation}
Moreover, the boundedness of $\|u\|_{L^\infty}$ directly gives
\begin{equation}\label{J3}
J_3\leq c_2\int_\Omega |\nabla  u|^4.
\end{equation}
Substituting \eqref{J1}-\eqref{J3} into \eqref{L4e-2}, and noting the facts $\gamma(v)\geq \gamma_1>0$ and $|\Delta u|\leq \sqrt{2}|D^2 u|$,  we have
\begin{equation*}\label{L4e-3}
\begin{split}
&\frac{1}{4}\frac{d}{dt}\int_\Omega |\nabla u|^4+\frac{\gamma_1}{2}\int_\Omega |\nabla |\nabla u|^2|^2 +\gamma_1\int_\Omega  |\nabla u|^2 |D^2 u|^2\\
&\leq \frac{1}{2}\int_{\partial \Omega} \gamma(v)|\nabla u|^2\frac{\partial |\nabla u|^2}{\partial \nu}dS+\frac{3}{2}\int_\Omega |\gamma'(v)||\nabla |\nabla u|^2||\nabla u|^2|\nabla v|\\
&\ \ \ \ +c_1\int_\Omega |\nabla u||\nabla |\nabla u|^2|(1+|\nabla u|)+c_1\int_\Omega |\nabla u|^2|\Delta u|(1+|\nabla u|)+c_2\int_\Omega |\nabla  u|^4\\
&\leq\frac{\gamma_1}{4}\int_\Omega  |\nabla |\nabla u|^2|^2 +\frac{\gamma_1}{2}\int_\Omega |\nabla u|^2 |D^2 u|^2+c_3\int_\Omega |\nabla  u|^4+c_4,\\
\end{split}
\end{equation*}
which leads to
\begin{equation}\label{L4e-3}
\begin{split}
\frac{d}{dt}\int_\Omega |\nabla u|^4+\gamma_1\int_\Omega |\nabla |\nabla u|^2|^2 +2\gamma_1\int_\Omega  |\nabla u|^2 |D^2 u|^2
&\leq4c_3\int_\Omega |\nabla  u|^4+4c_4.\\
\end{split}
\end{equation}
{\color{black}
On the other hand, using the boundedness of $\|u\|_{L^\infty}$ and the fact $|\Delta u|\leq \sqrt{2}|D^2 u|$ again, we have
\begin{equation*}
\begin{split}
\left(\frac{3}{2}+4c_3\right)\int_\Omega|\nabla u|^4
&=\left(\frac{3}{2}+4c_3\right)\int_\Omega |\nabla u|^2\nabla u\cdot\nabla u\\
&=-\left(\frac{3}{2}+4c_3\right)\int_\Omega u\nabla |\nabla u|^2\cdot\nabla u-\left(\frac{3}{2}+4c_3\right)\int_\Omega u|\nabla u|^2\Delta u\\
&\leq \gamma_1\int_\Omega |\nabla |\nabla u|^2|^2 +2\gamma_1\int_\Omega  |\nabla u|^2 |D^2 u|^2+c_5\int_\Omega|\nabla u|^2\\
&\leq \gamma_1\int_\Omega |\nabla |\nabla u|^2|^2 +2\gamma_1\int_\Omega  |\nabla u|^2 |D^2 u|^2+\frac{1}{2}\int_\Omega|\nabla u|^4+c_6,\\
\end{split}
\end{equation*}
which substituted into \eqref{L4e-3} gives}
\begin{equation}\label{L4e-4}
\frac{d}{dt}\int_\Omega |\nabla u|^4+\int_\Omega |\nabla u|^4\leq c_7.
\end{equation}
Then applying the Gronwall's inequality to \eqref{L4e-4} yields \eqref{L4e-1} and the proof is completed.

\end{proof}
\begin{lemma}\label{LT6}
Suppose $\mu>\frac{K_0}{16}$, and let $(u,v)$ be the global classical solution of the system \eqref{1-1*}.  Then there exists a constants $C>0$ such that for all $t>0$
\begin{equation}\label{LT6-1}
\|u(\cdot,t)-1\|_{L^\infty}\leq C e^{-\frac{\delta_*}{6} t},
\end{equation}
and
\begin{equation}\label{LT6-2}
\ \ \ \|v(\cdot,t)-1\|_{L^\infty}\leq C e^{-\frac{\delta_*}{6} t}.
\end{equation}

\end{lemma}
\begin{proof}
Using the Gagliardo-Nirenberg inequality, \eqref{LT5-1} and \eqref{L4e-1}, we have
\begin{equation*}\label{vc-1}
\begin{split}
\|u-1\|_{L^\infty}
&\leq c_1\|\nabla u\|_{L^4}^\frac{2}{3}\|u-1\|_{L^2}^\frac{1}{3}+c_1\|u-1\|_{L^2}\\
&\leq c_2 e^{-\frac{\delta_*}{6} t}+c_2e^{-\frac{\delta_*}{2} t}\\
&\leq 2c_2 e^{-\frac{\delta_*}{6} t},
\end{split}
\end{equation*}
which gives \eqref{LT6-1}.  \eqref{LT6-2} follows from \eqref{LT6-1} due to \eqref{vc}. This competes the proof of  Lemma \ref{LT6}.

\end{proof}
\begin{proof}[Proof of Theorem \ref{GB}]Theorem \ref{GB} is a consequence results by combining Lemma \ref{L8} and Lemma \ref{LT6}.
\end{proof}

\bigbreak

\noindent \textbf{Acknowledgement}.
We are grateful to the referee for several helpful comments improving our results. The research of H.Y. Jin was supported by the NSF of China (No. 11871226), Guangdong Basic and Applied Basic Research Foundation (No. 2020A1515010140), Guangzhou Science and Technology Program and the Fundamental Research Funds for the Central Universities. The research of Z.A. Wang was supported by the Hong Kong RGC GRF grant 15303019 (Project ID P0030816).

\end{document}